\documentclass[11pt,a4paper]{amsart}
\usepackage{amsmath,amsthm,amssymb,latexsym}
\usepackage{graphicx}

\newcommand{\bd}{{\mathbb{D}}}

\newcommand{\br}{{\mathbb{R}}}

\newcommand{\bc}{{\mathbb{C}}}

\newcommand{\bt}{{\mathbb{T}}}

\newcommand{\css}{{\mathcal{S}}}

\renewcommand{\a}{\alpha}
\renewcommand{\b}{\beta}
\renewcommand{\l}{\lambda}
\renewcommand{\ll}{\Lambda}
\newcommand{\s}{\sigma}

\renewcommand{\r}{\rho}

\newcommand{\p}{\varphi}
\renewcommand{\t}{\tau}
\renewcommand{\th}{\theta}
\renewcommand{\d}{\delta}

\renewcommand{\o}{\omega}
\newcommand{\oo}{{\Omega}}
\newcommand{\g}{\gamma}
\renewcommand{\gg}{\Gamma}
\newcommand{\ep}{\varepsilon}
\newcommand{\z}{\zeta}

\newcommand{\nt}{\noindent}

\newcommand{\ovl}{\overline}

\newcommand{\lp}{\left(}
\newcommand{\rp}{\right)}
\newcommand{\lt}{\left\{}
\newcommand{\rt}{\right\}}

\allowdisplaybreaks
\numberwithin{equation}{section}

\newtheorem{theorem}{Theorem}[section]

\newtheorem{corollary}[theorem]{Corollary}
\newtheorem{proposition}[theorem]{Proposition}

\theoremstyle{definition}

\newtheorem{remark}[theorem]{Remark}

\begin{document}

\title[Blaschke-type condition]
{On a Blaschke-type condition for subharmonic functions with two sets of singularities on the boundary}

\author[S. Favorov]{S. Favorov}

\address{Karazin Kharkiv National University,
4 Svobody sq., Kharkiv 61022, Ukraine}
\email{sfavorov@gmail.com}

\author[L. Golinskii]{L. Golinskii}

\address{B. Verkin Institute for Low Temperature Physics and
Engineering, 47 Science ave., Kharkiv 61103, Ukraine}
\email{golinskii@ilt.kharkov.ua}

\date{\today}

\subjclass[2010]{}

\maketitle

\begin{abstract}
Given two compact sets, $E$ and $F$, on the unit circle, we study the class of subharmonic functions on the unit disk which can
grow at the direction of $E$ and $F$ (sets of singularities) at different rate. The main result concerns the Blaschke-type condition
for the Riesz measure of such functions. The optimal character of such condition is demonstrated.
\end{abstract}

\vspace{0.5cm}
\begin{center}
{\large\it To Victor Katsnelson on occasion of his 75-th anniversary}
\end{center}

\vspace{0.5cm}
\section*{Introduction}
\label{s0}

In 1915, around a century ago, a seminal paper (6-pages note!) \cite{bla} by W.~Blaschke came out. A condition widely known nowadays as the
{\it Blaschke condition} for zeros of bounded analytic functions on the unit disk $\bd$
\begin{equation}\label{fundbla}
\sum_{\z\in Z(f)} (1-|\z|)<\infty
\end{equation}
was announced in this gem of Complex Analysis. Around 50 years ago both the authors learned about the Blaschke condition
from VK, being his graduate students.

It is not our intention reviewing a vast literature with various refinements
and far reaching extensions of \eqref{fundbla}, which appeared since then. We mention only that in all such extensions the majorants
of the (unbounded) functions in question were {\it radial}, that is, they depended on the {\it absolute value} of the argument. In other words,
the function was allowed to grow uniformly near the unit circle $\bt$.

We came across functions with {\it non-radial} growth for the first time in a result of Killip and Simon \cite[Theorem 2.8]{KiSi03}, where this bound looked
\begin{equation}\label{kilsim}
\log |L(z,J)|\le \frac{C}{|z^2-1|^2}\,, \quad z\in\bd.
\end{equation}
In the spectral theory setting of this paper the function $L$ (the perturbation determinant) turned out to belong to the Nevanlinna class, so
its zeros satisfied \eqref{fundbla}.

The question arose naturally what one could say about the zeros of a generic function which can grow at the directions toward some selected compact
sets on $\bt$ (we refer to these sets as the {\it sets of singularities}). For example, in \eqref{kilsim} this set is $E=\{\pm1\}$.
The study of such functions and their zero sets was initiated in \cite{bgk1, bgk2} for analytic functions, and in \cite{fg1, fg2} for subharmonic functions on
$\bd$. To remain closer to the main subject of our paper -- functions with two sets of singularities on $\bt$ -- we mention
two results from the preceding papers.

Given a compact set $F\subset\bt$, denote by $\rho_F(w)$ the Euclidian distance from a point $w\in\bc$ to the set $F$.
Recall the following quantitative characteristic of $F$ known as the {\it Ahern--Clark type} \cite{ahcla}
$$ \a(F):=\sup\{\a\in\br: \ m\bigl(\z\in\bt: \rho_F(\z)<x\bigr)=O(x^\a), \quad x\to +0\}, $$
$m(A)$ is the normalized Lebesgue measure of a set $A$.

The first aforementioned result is a particular case of \cite[Theorem 0.3]{bgk2}.

{\bf Theorem A}. {\it Given a compact set $F\subset\bt$, let an analytic function $f$ on $\bd$, $|f(0)|=1$, satisfy the growth condition
\begin{equation*}
\log|f(z)|\le \frac{M}{(1-|z|)^p\,\rho_F^q(z)}\,, \qquad z\in\bd, \quad M,p,q>0.
\end{equation*}
Then for each $\ep>0$ there is a positive number $C=C(F,p,q,\ep)$ so that the Blaschke-type condition holds for the zero set $Z(f)$ of $f$}
\begin{equation*}
\sum_{\z\in Z(f)} (1-|\z|)^{p+1+\ep}\,\rho_F^{(q-\a(F)+\ep)_+}(\z)\le CM, \qquad (x)_+:=\max(x,0).
\end{equation*}

As it was pointed out in \cite{fg1}, the natural setting of the problem in question is the set of subharmonic functions of special growth. The analogue of
the Blaschke condition involves then the Riesz measure (generalized Laplacian) of the corresponding function.

The second result is a particular case $n=2$ of \cite[Theorem 5]{fg2}.

Let $E$ and $F$ be two arbitrary compact sets on $\bt$. We define a class $\css_{p,q}(E,F)$ of subharmonic on $\bd$ functions $v$, which satisfy
\begin{equation}\label{growth}
v(z)\le \frac{M}{\rho_E^p(z)\,\rho_F^q(z)}\,,  \qquad M,p,q>0.
\end{equation}

{\bf Theorem B}. {\it Given two disjoint compact sets $E,F\subset\bt$, let a subharmonic function $v\in\css_{p,q}(E,F)$.
Then for each $\ep>0$ the following Blaschke-type condition holds for the Riesz measure $\mu$ of $v$}
\begin{equation*}
\int_\bd (1-|\l|)\,\rho_E^{(p-\a(E)+\ep)_+}(\l)\,\rho_F^{(q-\a(F)+\ep)_+}(\l)\,\mu(d\l)<\infty.
\end{equation*}

Both the above results actually deal with two sets of singularities, and each case is extreme in a sense. Precisely, such sets are $E=\bt$ and $F$ in Theorem A, and
the {\it disjoint} sets $E$ and $F$ in Theorem B. The goal of this paper is to study the case of two generic compact sets which come up as the sets of
singularities of a subharmonic function $v$ subject to some special growth condition.

\smallskip

We impose certain restrictions on $E$ and $F$ in the form of ``integrability'' of the products
\begin{equation}\label{int3}
\bigl\|\rho_E^{-a}\rho_F^{-b}\bigr\|_1=\int_\bt \frac{m(d\z)}{\rho_E^a(\z)\rho_F^b(\z)}<\infty, \qquad a,b\ge0.
\end{equation}
Here is our main result.
\footnote{The case of more general conditions on a function $v$ and its associated measure was considered in the papers \cite{HA}, \cite{HT}, 
but these conditions do not look as clear as ours.}

\begin{theorem}\label{auxth}
Given two compact sets $E$ and $F$ on $\bt$ subject to $\eqref{int3}$, let a subharmonic function $v$, $v(0)\ge0$,
with the Riesz measure $\mu$, belong to $\css_{p,q}(E,F)$.

\noindent
$(i)$. If both $0\le a<p$ and $0\le b<q$ hold, then
for each $\ep>0$ there is a constant $C=C(p,q,a,b,\ep)$ so that
\begin{equation}\label{bla9}
\int_\bd \rho_E^{p-a+\ep}(\l)\,\rho_F^{q-b+\ep}(\l)(1-|\l|)\,\mu(d\l)\le CM\,\bigl\|\rho_E^{-a}\,\rho_F^{-b}\bigr\|_1.
\end{equation}
$(ii)$. If $0\le a<p$, $b\ge q$ \ $(0\le b<q$, $a\ge p)$, then for each $\ep>0$ there is a constant $C=C(p,q,a,\ep)$ \
$(C=C(p,q,b,\ep))$ so that
\begin{equation}\label{blaii}
\begin{split}
\int_\bd \rho_E^{p-a+\ep}(\l)\,(1-|\l|)\,\mu(d\l) &\le CM\,\bigl\|\rho_E^{-a}\rho_F^{-b}\bigr\|_1, \\
\Bigl( \int_\bd \rho_F^{q-b+\ep}(\l)\,(1-|\l|)\,\mu(d\l) &\le CM\,\bigl\|\rho_E^{-a}\rho_F^{-b}\bigr\|_1. \Bigr)
\end{split}
\end{equation}
\end{theorem}

\smallskip

The procedure we suggest for solving the problem under consideration is pursued in three steps.

Step 1. Given a function $v\in\css_{p,q}(E,F)$, we find a domain $\oo\subset\bd$ so that $v$ has a harmonic majorant, i.e.,
the harmonic function $U$ exists with $v\le U$ on $\oo$. By the Riesz representation, see, e.g., \cite[Theorem 4.5.4]{Ran},
which will feature prominently in what follows,
\begin{equation}\label{subrepr}
v(z)=u(z)-\int_\oo G_\oo(z,\l)\,\mu(d\l), \qquad z\in\oo.
\end{equation}
Here $u$ is the least harmonic majorant for $v$, $\mu$ the Riesz measure of $v$, $G_\oo$ the Green's function for $\oo$
$$ G_\oo(z,\l):=\log\frac1{|z-\l|}-h_\oo(z,\l), \qquad z,\l\in\oo, $$
$h_\oo$ is the solution to the Dirichlet problem on $\oo$ for the boundary value
$$ h_\oo(z,\xi)=\log\frac1{|z-\xi|}\,, \qquad \xi\in\partial\oo. $$
If $\oo$ contains the origin, and $v(0)\ge0$, we have from \eqref{subrepr} with $z=0$
\begin{equation}\label{blash1}
\int_\oo G_\oo(0,\l)\,\mu(d\l)\le u(0)\le U(0).
\end{equation}

Step 2. We apply the lower bound for the Green's function of the type
$$ G_\oo(0,\l)\ge c(1-|\l|), \qquad \l\in\oo'\subset\oo $$
to obtain
$$ \int_{\oo'} (1-|\l|)\,\mu(d\l)\le U(0). $$

Step 3. To go over to the integration over the whole unit disk, we invoke a new two-dimensional version of the well-known
``layer cake representation'' (LCR) theorem, see Proposition \ref{pr5}.

In the simplest case when $\oo=\bd$ (see Theorem \ref{th1} below) the Green's function~is
$$ G_{\bd}(z,\l)=\log\left|\frac{1-\bar\l z}{z-\l}\right|, $$
so we come to the Blaschke condition for $\mu$ of the form
\begin{equation}\label{blashdisk}
\int_\bd (1-|\l|)\,\mu(d\l)\le\int_\bd \log\frac1{|\l|}\,\mu(d\l)\le U(0)
\end{equation}
in one step.

\smallskip

We proceed as follows. In Section \ref{s1} we gather a collection of auxiliary facts on the harmonic measure and majorants,
the bounds from below for the Green's function and LCR theorems. The main result is proved in Section~\ref{s2}. We also demonstrate its
optimal character in Theorem \ref{inverse}.

\section{Preliminaries}
\label{s1}

\subsection{Bounds for the harmonic measure}
Let $\g=[e^{i\th_1}, e^{i\th_2}]$ be a closed arc on the unit circle $\bt$. For the harmonic measure of this
arc with respect to the unit disk $\bd$ the explicit expression is known \cite[p. 26]{Ga-Ma}
\begin{equation*}
\o(\l,\g;\bd)=\frac{2\a-(\th_2-\th_1)}{2\pi}\,, \qquad \l\in\bd,
\end{equation*}
where $\a$ is the angle subtended at $\l$ by the arc $\g$.

Let $\z'\in\bt$, and $0<t<1$. We put
\begin{equation}\label{arcs}
\begin{split}
\g &=\g_t(\z'):=\{\z\in\bt: \ |\z-\z'|\le t\}, \\
\gg &=\gg_t(\z'):=\{z\in\bd: \ |z-\z'|=t\}.
\end{split}
\end{equation}
It is clear, that $\o$ is constant on $\gg$ (it is constant on each arc of a circle that passes through the endpoints of $\g$).
An elementary geometry provides the formula
$$ \o(\l,\g_t(\z'); \bd)=\frac12-\frac1{\pi}\,\arcsin\frac{t}2\,, \qquad \l\in\gg_t(\z'). $$
So, there is a uniform bound from below for the harmonic measure of $\g$ on~$\gg$
\begin{equation}\label{belarc}
\o(\l,\g_t(\z'); \bd)\ge \frac13\,, \qquad \l\in\gg_t(\z').
\end{equation}

To proceed further, given a compact set $K\subset\bt$, denote by
$$ \rho(w)=\rho_K(w):={\rm dist}(w,K), \quad w\in\bc, $$
the Euclidian distance from $w$ to $K$. Consider the sets on the unit circle
\begin{equation}\label{neib}
K_t :=\{\z\in\bt: \ \rho_K(\z)\le t\}, \quad K_t':=\{\z\in\bt: \ \rho_K(\z)\ge t\}=\ovl{\bt\backslash K_t},
\end{equation}
and the set in $\bd$
\begin{equation}\label{niebdis}
\gg_t(K):=\{z\in\bd: \ \rho_K(z)=t\}.
\end{equation}
Note that $K_t$ and $K_t'$ are finite unions of disjoint closed arcs.

For each $\l\in\gg_t(K)$ there is $\z'\in K$, such that $|\l-\z'|=\rho_K(\l)=t$, so $\l\in\gg_t(\z')$. If follows from relation \eqref{belarc} that
$$ \o(\l,\g_t(\z'); \bd)\ge \frac13, \qquad \l\in\gg_t(K). $$
But, by definition, $K_t\supset\g_t(\z')$ for each $\z'\in K$, so monotonicity of the harmonic measure yields
$$ \o(\l,K_t; \bd)\ge\o(\l,\g_t(\z'); \bd). $$

\begin{proposition}\label{pr1}
Given a compact set $K\subset\bt$, let $K_t$ be its closed neighborhood $\eqref{neib}$. Then
\begin{equation}\label{belgen}
\o(\l,K_t; \bd)\ge\frac13, \qquad \l\in\gg_t(K).
\end{equation}
\end{proposition}

\smallskip

Let us now turn to the upper bounds for the harmonic measure of $K_t$.
For a compact set $K$ on $\bt$ and $0<t<1$, the open set
\begin{equation}\label{set1}
D_{t}(K):=\{w\in\bd: \ \rho_K(w)>~t\}
\end{equation}
can be disconnected even for simple $K$. We denote by $\oo_t(K)$ the connected component of $D_t(K)$ that contains the origin.
Clearly, $\oo_t(K)=\emptyset$ for $t\ge1$.

In view of connectedness, it is easy to verify that $\oo_{t}(K)\supset\oo_{\tau}(K)$ for $\tau>t$. It is also important that
\begin{equation}\label{bound}
\partial\oo_t(K)\subset \partial D_t(K)=\gg_t(K)\cup K_t'.
\end{equation}

The following result will be helpful later on.

\begin{proposition}\label{pr2}
Given a compact set $K\subset\bt$, and $s>0$, one has
\begin{equation}\label{inclus}
\{w\in\bd: \ \rho_K(w)>s\}\subset\oo_{s/2}(K).
\end{equation}
\end{proposition}
\begin{proof}
Clearly,
$$ \{w\in\bd: \ \rho_K(w)>s\}\subset\lt w\in\bd: \ \rho_K(w)>\frac{s}2\rt, $$
and we wish to show that the set on the left side is actually a subset of the connected component of the set on the right side
that contains the origin. The argument relies on a simple inequality, which we apply repeatedly throughout the paper
\begin{equation}\label{ineq}
\rho_K(z)\le 2\rho_K(rz), \qquad z\in\ovl\bd, \quad 0\le r\le1.
\end{equation}
Indeed, by the triangle inequality $\rho_K(z)\le \rho_K(rz)+|rz-z|$, and so
$$ \rho_K(z)\le \rho_K(rz)+(1-|rz|)\le 2\rho_K(rz), $$
as claimed.

It follows from \eqref{ineq} that $\rho_K(rz)>s/2$ for all $0\le r\le1$ as soon as $\rho_K(z)>s$.
In other words, the whole closed interval
$$ [0,z]\subset \lt\rho_K(w)>s/2\rt, $$
and so $z\in\oo_{s/2}(K)$, as needed.
\end{proof}

\begin{proposition}\label{ubound}
Given a number $l\in (0,1)$, put $k:=\pi l^{-1}+1$. Then the following inequality holds for $t<k^{-1}$
\begin{equation}\label{ubouhm}
\o(\l,K_t;\bd)\le \frac{l}{t}\,(1-|\l|), \qquad \l\in\oo_{kt}(K).
\end{equation}
\end{proposition}
\begin{proof}
If $1-|\l|>t l^{-1}$, inequality \eqref{ubouhm} obviously holds. So we assume in what follows that
\begin{equation}\label{boun2}
1-|\l|\le\frac{t}{l}\,, \qquad |\l|\ge 1-\frac{t}{l}>1-\frac1{kl}.
\end{equation}

For $\l=|\l|e^{\i\th}\in\oo_{kt}(K)$, and $\z=e^{i\p}\in K_t$, the Poisson integral representation for the harmonic measure reads
\begin{equation*}
\o(\l,K_t;\bd)=\int\limits_{K_t} \frac{1-|\l|^2}{|\z-\l|^2}\,m(d\z)=
\frac{(1-|\l|^2)}{2\pi}\,\int\limits_{e^{i\p}\in K_t} \frac{d\p}{(1-|\l|)^2+4|\l|\sin^2\frac{\p-\th}2}\,.
\end{equation*}
Take $\z_1\in K$ such that $\rho_K(e^{i\th}) =|e^{i\th}-\z_1|$. Then, in view of \eqref{boun2},
\begin{equation}\label{boun4}
\begin{split}
\rho_K(e^{i\th}) &=|e^{i\th}-\z_1|=|\l-\z_1+e^{i\th}-\l| \\
&\ge \rho_K(\l)-(1-|\l|)>kt-(1-|\l|).
\end{split}
\end{equation}
Take $\z_2\in K$ such that $\rho_K(e^{i\p})=|e^{i\p}-\z_2|\le t$, so, by \eqref{boun4},
\begin{equation*}
\begin{split}
|\th-\p| &\ge \bigl|e^{i\th}-e^{i\p}\bigr|=\bigl|e^{i\th}-\z_2+\z_2-e^{i\p}\bigr| \\
&\ge\rho_K(e^{i\th})-t>(k-1)t-(1-|\l|).
\end{split}
\end{equation*}
Hence \eqref{boun2} implies
$$ \pi\ge |\th-\p|\ge (k-1)t-\frac{t}{l}=k_1t, \quad k_1:=k-1-\frac1{l}. $$
Going back to the Poisson integral, we see that
$$ \o(\l,K_t;\bd)\le \frac{1-|\l|}{4\pi |\l|}\,\int\limits_{k_1t\le|\p-\th|\le\pi}\,\frac{d\p}{\sin^2\frac{\p-\th}2}
\le \frac{\pi(1-|\l|)}{4|\l|}\,\int\limits_{k_1t\le|\p-\th|\le\pi}\,\frac{d\p}{(\p-\th)^2}\,, $$
or, in view of \eqref{boun2},
$$ \o(\l,K_t;\bd)\le \frac{\pi(1-|\l|)}{4}\,\frac{kl}{kl-1}\,\int_{k_1t}^\pi \frac{dx}{x^2}
\le \frac{\pi(1-|\l|)}{4t}\,\frac{kl}{kl-1}\,\frac1{k-1-l^{-1}}\,. $$
An elementary calculation shows that
$$ \frac{\pi}{4}\,\frac{kl}{kl-1}\,\frac1{k-1-l^{-1}}<l, $$
and \eqref{ubouhm} follows.
\end{proof}

\subsection{Lower bounds for Green's functions}
Under a {\it Green's function} of the domain $\oo_t(K)$ with singularity $z$ we mean a nonnegative function of the form
\begin{equation}\label{green}
G_t(z,\l)=G_{\oo_t(K)}(z,\l):=\log\frac1{|z-\l|}-h_t(z,\l), \qquad z,\l\in\oo_t(K),
\end{equation}
where $h_t$ is the solution to the Dirichlet problem on $\oo_t(K)$ for the boundary value
\begin{equation}\label{bounddir}
h_t(z,\xi)=\log\frac1{|z-\xi|}\,, \qquad \xi\in\partial\oo_t(K).
\end{equation}
Such function exists and is unique, as the boundary $\partial\oo_t(K)$ is a non-polar set, see, e.g., \cite{Ran}.
The problem we address here is to obtain a lower bound for $G_t(0,\cdot)$ in a smaller domain $\oo_\t(K)$
with an appropriate $\t>t$.

\begin{proposition}\label{pr3}
The Green's function $G_t(0,\cdot)$ for the domain $\oo_t(K)$ with singularity at the origin and $0<t<(12\pi+1)^{-1}$ admits the lower bound
\begin{equation}\label{greenboun}
G_t(0,\l) \ge\frac{1-|\l|}2, \qquad \l\in\oo_{(12\pi+1)t}(K).
\end{equation}
\end{proposition}
\begin{proof}
Since
$$ 1-|\xi|\le\rho_K(\xi)=t, \quad |\xi|\ge 1-t>\frac12, \qquad \xi\in\gg_t(K), $$
one has
$$ h_t(0,\xi)=\log\frac1{|\xi|}\le\log\frac1{1-t}\le 2t, \quad \xi\in\partial\oo_t(K)\cap\gg_t(K). $$
Next, $h_t(0,\xi)=0$ for $\xi\in\partial\oo_t(K)\cap K_t'$, so, by Proposition \ref{pr1} and the Maximum Principle,
$$ h_t(0,\l)\le 6t\o(\l,K_t;\bd), \qquad \l\in\oo_t(K). $$
Now, the upper bound \eqref{ubouhm} with $l=1/12$ and $k=12\pi+1$  yields
$$ h_t(0,\l)\le \frac{1-|\l|}2, \qquad \l\in\oo_{(12\pi+1)t}(K), $$
and so
$$ G_t(0,\l)\ge \log\frac1{|\l|}-\frac{1-|\l|}2\ge \frac{1-|\l|}2, \qquad \l\in\oo_{(12\pi+1)t}(K), $$
as needed.
\end{proof}

\smallskip

So far we have been dealing with one compact set $K$. Keeping in mind the main topic of the paper, consider the intersection
\begin{equation*}
D_{t,s}(E,F):=\{w\in\bd:\rho_E(w)>t, \rho_F(w)>s\}=D_t(E)\cap D_s(F),
\end{equation*}
where $E$ and $F$ are compact sets on the unit circle, $0<t,s<1$. Denote by $\oo_{t,s}$ the connected component of this open set
(or, that is the same, the connected component of $\oo_t(E)\cap\oo_s(F)$) so that $0\in\oo_{t,s}$. Clearly, $\oo_{t,s}=\emptyset$
for $\max\{t,s\}\ge1$. It is not hard to check that
\begin{equation}\label{boundfor2}
\begin{split}
\partial\oo_{t,s} &\subset W_1\cup W_2\cup W_3, \qquad W_3:=E_t'\cap F_s', \\
W_1 &:=\{w\in\bd:\rho_E(w)=t,\ \rho_F(w)\ge s\}, \\
W_2 &:=\{w\in\bd:\rho_E(w)\ge t,\ \rho_F(w)=s\}.
\end{split}
\end{equation}
In particular,
\begin{equation}\label{boundfor2.1}
\partial\oo_{t,s} \subset \gg_t(E)\cup\gg_s(F)\cup (E_t'\cap F_s').
\end{equation}

The inclusion
\begin{equation}\label{inclus2}
D_{t,s}(E,F)\subset\oo_{t/2,s/2},
\end{equation}
can be verified in exactly the same way as \eqref{inclus} in Proposition \ref{pr2}.

We complete with the lower bound for the Green's function $G_{t,s}:=G_{\oo_{t,s}}$.

\begin{proposition}\label{pr4}
The Green's function $G_{t,s}(0,\cdot)$ for the domain $\oo_{t,s}$ with singularity at the origin and $0<t,s<(24\pi+1)^{-1}$ admits the lower bound
\begin{equation}\label{greenboun2}
G_{t,s}(0,\l) \ge\frac{1-|\l|}2, \qquad \l\in\oo_{(24\pi+1)t,(24\pi+1)s}.
\end{equation}
\end{proposition}
\begin{proof}
We follow the argument from the proof of Proposition \ref{pr3}. Write
$$ G_{t,s}(0,\l)=\log\frac1{|\l|}-h(\l), \quad h(\z)=\log\frac1{|\z|}, \quad \z\in\partial\oo_{t,s}, $$
so $h(\z)=0$ for $\z\in\partial\oo_{t,s}\cap\bt$. Since
$$ |\z|\ge 1-t>1/2, \quad \z\in\gg_t(E), \qquad |\z|\ge 1-s>1/2, \quad \z\in\gg_s(F), $$
we have
\begin{equation}\label{onbound}
\begin{split}
\log\frac1{|\z|} &\le 2(1-|\z|)\le 2t, \qquad \z\in\gg_t(E), \\
\log\frac1{|\z|} &\le 2(1-|\z|)\le 2s, \qquad \z\in\gg_s(F).
\end{split}
\end{equation}
In view of \eqref{boundfor2.1}, \eqref{onbound} and Proposition \ref{pr1}, it follows from the Maximum Principle that
$$ h(\l)\le 6t\o(\l,E_t;\bd)+6s\o(\l,F_s;\bd), \qquad \l\in\oo_{t,s}. $$
We apply the upper bound for the harmonic measure \eqref{ubouhm}
$$ t\o(\l,E_t;\bd)+s\o(\l,F_s;\bd)\le 2l(1-|\l|), \qquad \l\in\oo_{t,s}, $$
so for $l=1/24$, $k=24\pi+1$ we come to
$$ h(\l)\le\frac{1-|\l|}2\le\frac12\,\log\frac1{|\l|} \ \Rightarrow \ G_{t,s}(0,\l)\ge\frac12\,\log\frac1{|\l|}\ge\frac{1-|\l|}2, $$
as claimed.
\end{proof}

\subsection{Harmonic majorant}
The result below concerns particular subharmonic functions and their harmonic majorants.
\smallskip

\if{Given a compact set $K$ on $\bt$ and $a>0$, assume that $\rho_K^{-a}\in L^1(\bt)$. Then the function
\begin{equation}\label{specsub1}
 v_{K,a}(z):=\frac1{\rho_K^a(z)}\,, \qquad z\in\bd,
 \end{equation}
is subharmonic and admits the harmonic majorant
\begin{equation}\label{hama1}
v_{K,a}(z)\le P_{K,a}(z):=\int_\bt\frac{1-|z|^2}{|\z-z|^2}\,\frac{m(d\z)}{\rho_K^a(\z)}\,.
\end{equation}
}\fi
\begin{proposition}\label{harmaj}
Given two compact sets $E$ and $F$ on the unit circle, and $a,b\ge0$, assume that $\rho_E^{-a}\rho_F^{-b}\in L^1(\bt)$.
Then the function
\begin{equation}\label{specsub2}
v_{a,b}(z):=\frac1{\rho_E^a(z)\,\rho_F^b(z)}\,, \qquad z\in\bd,
\end{equation}
is subharmonic and admits the harmonic majorant
\begin{equation}\label{hama2}
v_{a,b}(z)\le P_{a,b}(z):=\int_\bt\frac{1-|z|^2}{|\z-z|^2}\,\frac{m(d\z)}{\rho_E^a(\z)\,\rho_F^b(\z)}\,.
\end{equation}
\end{proposition}
\begin{proof}
The case $a=b=0$ is trivial, so let $a+b>0$. By \cite[Theorem 2.4.7]{Ran}, the function
$$ v_{a,b}(z)=\sup_{\xi\in E,\,\eta\in F} |(z-\xi)^{-a}(z-\eta)^{-b}| $$
is subharmonic. The inequality \eqref{ineq} implies
\begin{equation}\label{ineq2}
v_{a,b}(r\z)\le 2^{a+b}v_{a,b}(\z), \qquad \z\in\bt, \quad 0\le r\le 1.
\end{equation}

The standard Maximum Principle states that
$$ v_{a,b}(rz)\le\int_\bt\frac{1-|z|^2}{|\z-z|^2}\,v_{a,b}(r\z)\,m(d\z), \qquad z\in\bd, \quad r<1. $$
The bound \eqref{hama2} is now immediate from the latter inequality as $r\to 1-0$
due to \eqref{ineq2} and the Lebesgue Dominated convergence theorem.
\end{proof}

\begin{remark}
As a matter of fact, $P_{a,b}$ is the {\it least} harmonic majorant for $v_{a,b}$, see, e.g., \cite[pp.36-37]{Gar}.
\end{remark}

\subsection{Layer cake representation}
A key ingredient in our argument is the fundamental result in Analysis, known as the ``layer cake representation'' (LCR) see, e.g.,
\cite[Theorem 1.13]{Li-Lo}.

{\bf Theorem LCR}. Let $(\ll,\nu)$ be a measure space, and $h\ge0$ a measurable function on $\ll$. Then for $c>0$ the equality holds
\begin{equation}\label{lcr}
\int_\ll h^c(\tau)\,\nu(d\tau)=c\int_0^\infty x^{c-1}\nu(\{\tau: h(\tau)>x\})\,dx.
\end{equation}

In what follows we make use of the two-dimensional analogue of this result.

\begin{proposition}\label{pr5}
Let $f,g\ge0$ be measurable functions on the measure space $(\ll,\s)$, and $\a,\b>0$. Then
\begin{equation}\label{lcr2d}
\begin{split}
I & := \int_\ll f^\a(\tau)g^\b(\tau)\,\s(d\tau) \\
& = \a\b\int_0^\infty\int_0^\infty x^{\a-1}y^{\b-1}\s(\{\tau: f(\tau)>x,\, g(\tau)>y\})\,dx\,dy.
\end{split}
\end{equation}
\end{proposition}
\begin{proof}
We apply the LCR \eqref{lcr} twice. Put $\nu(d\tau):=g^\b\s(d\tau)$, so
$$ I=\int_\ll f^\a(\tau)\,\nu(d\tau)=\a\int_0^\infty x^{\a-1}\nu(\{\tau: f(\tau)>x\})\,dx. $$
Write $\ll_x:=\{\tau: f(\tau)>x\}$, and apply \eqref{lcr} once again
\begin{equation*}
\begin{split}
\nu(\ll_x) &=\int_{\ll_x} g^\b(\tau)\,\s(d\tau)=\b\int_0^\infty y^{\b-1}\nu(\{\tau\in\ll_x: g(\tau)>y\})\,dy \\
& = \b\int_0^\infty y^{\b-1}\nu(\{\tau\in\ll: f(\tau)>x,\,g(\tau)>y\})\,dy,
\end{split}
\end{equation*}
so Fubini's theorem completes the proof.
\end{proof}

\section{Problem with two compact sets}
\label{s2}

Let us go back to our main problem concerning the Blaschke-type condition for the Riesz measure of the subharmonic function
which can grow at the direction of two sets of singularities on the unit circle.

As a warm-up, we prove the following result.

\begin{theorem}\label{th1}
Assume that $E$ and $F$ are two compact sets on $\bt$ so that \eqref{int3} holds with $a=p$, $b=q$. For each subharmonic function
$v\in\css_{p,q}(E,F)$, $v(0)\ge0$, with the Riesz measure $\mu$, the Blaschke condition holds
\begin{equation}\label{lightbl}
\int_\bd (1-|\l|)\,\mu(d\l)\le M\|\rho_E^{-p}\rho_F^{-q}\|_1.
\end{equation}
\end{theorem}
\begin{proof}
By Proposition \ref{harmaj}, $v$ admits the harmonic majorant $U=MP_{p,q}$ with $U(0)=M\|\rho_E^{-p}\rho_F^{-q}\|_1$.
Relation \eqref{blashdisk} completes the proof.
\end{proof}

The case when $\min(p,q)=0$, so we actually have one compact set, was elaborated in \cite{fg1}.

The main result of the paper, Theorem \ref{auxth}, concerns the rest of the values for $a$ and $b$, that is, either $0\le a<p$ or $0\le b<q$.

\smallskip

\nt {\it Proof of Theorem \ref{auxth}.} \newline
(i).\,We proceed in three steps, following the procedure outlined in Introduction.

Step 1. Write the hypothesis \eqref{growth} as
$$ v(z)\le v_{a,b}(z)\,\frac{M}{\rho_E^{p-a}(z)\rho_F^{q-b}(z)}\,, \quad z\in\bd. $$
In view of \eqref{boundfor2}, Proposition \ref{harmaj}, and the Maximum Principle, we come to the bound
$$ v(z)\le U(z)=P_{a,b}(z)\,\frac{M}{t^{p-a}\,s^{q-b}}\,, \quad z\in\oo_{t,s}. $$

Step 2. Relation \eqref{blash1} now reads
$$ \int\limits_{\oo_{t,s}} G_{t,s}(0,\l)\,\mu(d\l)\le u(0)\le U(0)=\frac{M}{t^{p-a}\,s^{q-b}}\,\bigl\|\rho_E^{-a}\,\rho_F^{-b}\bigr\|_1. $$
By Proposition \ref{pr4} with $\kappa=24\pi+1$, one has
$$ \int\limits_{\oo_{\kappa t,\kappa s}} (1-|\l|)\,\mu(d\l)\le\frac{2M}{t^{p-a}\,s^{q-b}}\,\bigl\|\rho_E^{-a}\,\rho_F^{-b}\bigr\|_1. $$
By \eqref{inclus2}, $D_{2\kappa t, 2\kappa s}(E,F)\subset \oo_{\kappa t,\kappa s}$, so putting
$$ \xi:=2\kappa t, \quad \eta:=2\kappa s, \qquad 0\le t,s\le\frac1{2\kappa}, $$
we end up with the bound
\begin{equation}\label{bound9}
\int\limits_{D_{\xi,\eta}(E,F)} (1-|\l|)\,\mu(d\l)\le 2(2\kappa)^{p+q-a-b}\,\frac{M}{\xi^{p-a}\,\eta^{q-b}}\,\bigl\|\rho_E^{-a}\,\rho_F^{-b}\bigr\|_1.
\end{equation}

Step 3. The LCR theorem comes into play here. By Proposition \ref{pr5} with
$$ \ll=\bd, \ \s=(1-|\l|)\mu, \ f=\rho_E, \ g=\rho_F, \ \a=p-a+\ep, \ \b=q-b+\ep, $$
we see that
$$ \int_\bd \rho_E^\a(\l)\rho_F^\b(\l)\s(d\l)=\a\b\int_0^2\int_0^2\xi^{\a-1}\eta^{\b-1}\s\Bigl(\{\l: \rho_E(\l)>\xi, \rho_F(\l)>\eta\}\Bigr)d\xi d\eta. $$
But, due to \eqref{bound9},
\begin{equation*}
\begin{split}
\s\Bigl(\{\l: \rho_E(\l)>\xi, \rho_F(\l)>\eta\}\Bigr) &=\int\limits_{D_{\xi,\eta}(E,F)} (1-|\l|)\,\mu(d\l) \\
&\le \frac{CM}{\xi^{p-a}\,\eta^{q-b}}\,\bigl\|\rho_E^{-a}\,\rho_F^{-b}\bigr\|_1,
\end{split}
\end{equation*}
so, finally,
\begin{equation*}
\int_\bd \rho_E^\a(\l)\rho_F^\b(\l)(1-|\l|)\,\mu(d\l)\le \a\b\,CM\,\bigl\|\rho_E^{-a}\,\rho_F^{-b}\bigr\|_1\,\int_0^2 \xi^{\ep-1}d\xi\,
\int_0^2 \eta^{\ep-1}d\eta,
\end{equation*}
and the first statement is proved.

(ii). Assume now that $0\le a<p$ and $b\ge q$. The argument is the same but simpler, as we appeal to the domain $\oo_t(E)$ and the standard one-dimensional
LCR theorem \eqref{lcr}. Indeed, as in Step 1, we have
$$ v(z)\le U(z)=P_{a,b}(z)\,\frac{2^{b-q}M}{t^{p-a}}\,, \qquad z\in\oo_t(E). $$
Next, relation \eqref{blash1} provides
$$ \int\limits_{\oo_t(E)} G_t(0,\l)\,\mu(d\l)\le \frac{2^{b-q}M}{t^{p-a}}\,\bigl\|\rho_E^{-a}\,\rho_F^{-b}\bigr\|_1, $$
so, by Proposition \ref{pr3} with $\kappa=12\pi+1$,
$$ \int\limits_{\oo_{\kappa t}(E)} (1-|\l|)\,\mu(d\l)\le \frac{2^{b-q+1}M}{t^{p-a}}\,\bigl\|\rho_E^{-a}\,\rho_F^{-b}\bigr\|_1. $$
By \eqref{inclus}, $D_{2\kappa t}(E)\subset \oo_{\kappa t}(E)$, and so for $\xi=2\kappa t$ we have
$$ \int\limits_{D_{\xi}(E)} (1-|\l|)\,\mu(d\l)\le (2\kappa)^{p-a}\,\frac{2^{b-q+1}M}{\xi^{p-a}}\,\bigl\|\rho_E^{-a}\,\rho_F^{-b}\bigr\|_1. $$
An application of LCR theorem in the form \eqref{lcr} with
$$ \ll=\bd, \quad \nu(d\l)=(1-|\l|)\,\mu(d\l), \quad h=\rho_E, \quad c=p-a+\ep $$
leads to the first Blaschke-type condition in \eqref{blaii}. The proof of the second one is identical.  \hfill $\Box$

The case $a=b=0$ is important, for there are no integrability assumptions whatsoever.

\begin{corollary}
Given two compact sets $E$, $F$ on $\bt$, let a subharmonic function $v$, $v(0)\ge0$, belong to $\css_{p,q}(E,F)$. Then
for each $\ep>0$ there is a constant $C=C(p,q,\ep)$ so that
\begin{equation}\label{bla19}
\int_\bd \rho_E^{p+\ep}(\l)\,\rho_F^{q+\ep}(\l)(1-|\l|)\,\mu(d\l)\le CM.
\end{equation}
\end{corollary}

The results of Theorem \ref{auxth} can be extended to the case of $n$ compact sets on the unit circle with no additional efforts.

\begin{theorem}
Let $K_1,\dots,K_n$ be compact subsets of $\bt$, and let $v$ be a subharmonic function on $\bd$ with Riesz measure $\mu$ such that $v(0)\ge0$ and
 $$
   v(z)\le M\r_{K_1}^{-p_1}(z)\cdots \r_{K_n}^{-p_n}(z).\qquad z\in\bd.
 $$
Suppose that
$$
\r_{K_1}^{-a_1}(\z)\cdots \r_{K_n}^{-a_n}(\z)\in L^1(\bt)
$$
for some
$$
a_1<p_1,\dots,a_k<p_k,\ a_{k+1}=p_{k+1},\dots, a_n=p_n,\qquad k\le n.
$$
Then for each $\ep>0$ there is a constant $C=C(p_1,\dots,p_n,a_1,\dots,a_k,\ep)$ so that
$$
\int_\bd\r_{K_1}^{p_1-a_1+\ep}(\l)\cdots \r_{K_k}^{p_k-a_k+\ep}(\l)(1-|\l|)\,d\mu(\l) \le CM\|\r_{K_1}^{-a_1}(z)\cdots \r_{K_n}^{-a_n}(z)\|_1.
$$
\end{theorem}

In view of further applications, let us mention a special case of subharmonic functions $v=\log|f|$ with $f$ analytic on the unit disk.

\begin{corollary}
Let an analytic function $f$, $|f(0)|\ge1$, satisfy the growth condition
\begin{equation}\label{growthanal}
\log|f(z)|\le\frac{M}{\rho_E^p(z)\,\rho_F^q(z)}\,, \qquad M,p,q>0,
\end{equation}
with two compact sets $E$, $F$ on the unit circle. Assume that the relation $\eqref{int3}$ holds
for some $0\le a<p$ and $0\le b<q$. Then for each $\ep>0$ there is a constant $C=C(p,q,a,b,\ep)$ so that
\begin{equation*}
\sum_{n=1}^\infty (1-|\l_n|)\,\rho_E^{p-a+\ep}(\l_n)\,\rho_F^{q-b+\ep}(\l_n)\le CM\,\bigl\|\rho_E^{-a}\rho_F^{-b}\bigr\|_1,
\end{equation*}
where $\{\l_n\}_{n\ge1}$ are the zeros of $f$ counting multiplicity.
\end{corollary}

Next, we consider the situation where the integrability assumptions are imposed on $\rho_E$ and $\rho_F$ separately.
At the moment the following partial result is available.

\begin{proposition}\label{mainth}
Let a subharmonic function $v$, $v(0)\ge0$, belong to $\css_{p,q}(E,F)$. Assume that
\begin{equation}\label{sepint}
\bigl\|\rho_E^{-p}\bigr\|_1=\int_\bt \frac{m(d\z)}{\rho_E^p(\z)}<\infty, \quad
\bigl\|\rho_F^{-q}\bigr\|_1=\int_\bt \frac{m(d\z)}{\rho_F^q(\z)}<\infty.
\end{equation}
Let $p',q'\ge0$ be nonnegative constants such that $p'+q'>\max(p,q)$. Then there is a constant $C=C(p,q,p',q')$ so that
\begin{equation}
\int_\bd \rho_E^{p'}(\l)\,\rho_F^{q'}(\l)(1-|\l|)\,\mu(d\l)\le CM\lp \bigl\|\rho_E^{-p}\bigr\|_1+\bigl\|\rho_F^{-q}\bigr\|_1\rp.
\end{equation}
\end{proposition}
\begin{proof}
We focus on two particular cases of Theorem \ref{auxth}, namely, $a=0, b=q$ and $a=p, b=0$. The corresponding conditions \eqref{int3}
agree with \eqref{sepint}. It follows from \eqref{blaii} that
\begin{equation}\label{bla10}
\int_\bd \lp \rho_E^{p+\ep}(\l)+\rho_F^{q+\ep}(\l)\rp (1-|\l|)\,\mu(d\l)
\le CM\lp \bigl\|\rho_E^{-p}\bigr\|_1+\bigl\|\rho_F^{-q}\bigr\|_1\rp
\end{equation}
for arbitrary $\ep>0$. We choose this parameter from the condition
\begin{equation}\label{ep12}
0<\ep<\frac{p'+q'-\max(p,q)}2\,.
\end{equation}

The argument below is quite elementary. Let $0\le x,y\le2$. If $y\le x$, we have, by \eqref{ep12},
$$ x^{p'}\,y^{q'}=x^{p'+q'}\le 2^{p'+q'-p-\ep}\,x^{p+\ep}. $$
Similarly, for $x\le y$
$$ x^{p'}\,y^{q'}=y^{p'+q'}\le 2^{p'+q'-q-\ep}\,y^{q+\ep}. $$
So, for each $0\le x,y\le2$ we have
$$ x^{p'}\,y^{q'}\le C\lp x^{p+\ep}+y^{q+\ep}\rp, \qquad C=2^{p'+q'-\min(p,q)-2\ep}. $$
It remains only to put $x:=\rho_E(\l)$, $y:=\rho_F(\l)$ and make use of \eqref{bla10}. The proof is complete.
\end{proof}

\begin{remark}
In some instances the assumption $v(0)\ge0$ looks somewhat restrictive. If $-\infty<v(0)<0$, one can apply the above results to
the function $v_1(z)=v(z)-v(0)$, which belongs to the same class $\css_{p,q}(E,F)$. But now the constant $M$ depends on $v$, so
we actually have quantitative Blaschke-type conditions. For example,
\begin{equation}\label{quant}
\int_\bd \rho_E^{p-a+\ep}(\l)\,\rho_F^{q-b+\ep}(\l)(1-|\l|)\,\mu(d\l)<\infty
\end{equation}
holds in place of \eqref{bla9}.

If $v(0)=-\infty$, consider the Poisson integral in the disk $|z|<1/2$ with the boundary value $v$
$$
h(z):=\int_\bt\frac{1-|2z|^2}{|\z-2z|^2}\,v(\z/2)m(d\z).
$$
 Since $v$ is upper semicontinuous, we see that
$\lim_{z'\to z}h(z')\le v(z)$ for each $z$ with $|z|=1/2$ . By \cite[Theorem 2.4.5]{Ran} the function
 $$
v_1(z)=\left\{\begin{array}{ccc}\max(v(z),h(z))&\hbox{for}&|z|<1/2,
\\ v(z)&\hbox{for}& |z|\ge1/2\end{array}\right.
 $$
is subharmonic in $\bd$, and the restriction of its Riesz measure $\mu_1$ on the set
$\{z\in\bd:\,|z|>1/2\}$ agrees with $\mu$. Therefore,
$$
\int_\bd\r_E^{p-a}(\l)\r_F^{q-b}(\l)(1-|\l|)\,\Bigl(\mu_1(d\l)-\mu(d\l)\Bigr)=O(1).
$$
Since $v_1(0)>-\infty$, we again get the conclusions of the quantitative type similar to \eqref{quant}.
\end{remark}

\smallskip

We complete the paper with the result which demonstrates the optimal character of the bound \eqref{bla9} in Theorem \ref{auxth}.

Given a compact set $K\subset\bt$, define the value
$$ \d(K):=\sup\{d\ge0: \ \rho_K^{-d}\in L^1(\bt)\}. $$
It is clear that $0\le \d(K)\le1$. The equality
\begin{equation}\label{carl}
\int_\bt \frac{m(d\z)}{\rho_K^d(\z)}=2^{-d}+dI(d,K), \quad d>0, \qquad  I(d,K):=\int_0^2 \frac{m(K_t)}{t^{d+1}}\,dt
\end{equation}
follows easily from the LCR theorem \eqref{lcr}, see \cite[formula (15)]{fg1}. The characteristic $I(d,K)$ appeared already in \cite{car}.
So,
$$ \d(K)=\sup\{d\ge0: \ I(d,K)<\infty\}. $$

Choose two disjoint compact sets $E$ and $F$ with $\d(E)>0$, $\d(F)>0$. By the definition,
$$ \rho_E^{-\d(E)+\ep}\in L^1(\bt), \quad \rho_F^{-\d(F)+\ep}\in L^1(\bt), \quad 0<\ep<\min(\d(E), \d(F)), $$
and so \eqref{int3} holds with $a=\d(E)-\ep$, $b=\d(F)-\ep$ ($E$ and $F$ are disjoint). On the other hand,
$$ \rho_E^{-\d(E)-\ep}\notin L^1(\bt), \quad \rho_F^{-\d(F)-\ep}\notin L^1(\bt), $$
and, by \eqref{carl},
$$ I(\d(E)+\ep, E)=I(\d(F)+\ep, F)=+\infty. $$
In notation \eqref{set1} we take $t$, $s$ small enough so that
\begin{equation}\label{set1.1}
D_t^c(E)\cap D_s^c(F)=\emptyset, \quad D_t^c(K):=\bd\backslash D_t(K).
\end{equation}

Let $p>\d(E)$, $q>\d(F)$, and consider the function
$$ v_0(z)=v_E(z)+v_F(z)=\frac1{\rho_E^p(z)}+\frac1{\rho_F^q(z)}\,, \quad v_0\in\css_{p,q}(E,F). $$
Denote by $\mu_E$  $(\mu_F)$ the Riesz measure of the subharmonic function $v_E$  $(v_F)$. The result in Theorem \ref{auxth}, (i),
states that
$$ \int_\bd \rho_E^{p-\d(E)+2\ep}(\l)\,\rho_F^{q-\d(F)+2\ep}(\l)\,(1-|\l|)\,\mu_0(d\l)<\infty, \quad \mu_0:=\mu_E+\mu_F $$
is the Riesz measure of $v_0$ and $\ep>0$ is small enough.

\begin{theorem}\label{inverse}
For $0<\ep<\min(p-\d(E), q-\d(F))$ the relation holds
\begin{equation}\label{inver1}
I:=\int_\bd \rho_E^{p-\d(E)-\ep}(\l)\,\rho_F^{q-\d(F)-\ep}(\l)\,(1-|\l|)\,\mu_0(d\l)=+\infty.
\end{equation}
\end{theorem}
\begin{proof}
We bound the integral $I$ from below in a few steps. Clearly,
\begin{equation*}
\begin{split}
I &\ge \int\limits_{D_t^c(E)} \rho_E^{p-\d(E)-\ep}(\l)\,\rho_F^{q-\d(F)-\ep}(\l)\,(1-|\l|)\,\mu_0(d\l) \\
  &\ge \int\limits_{D_t^c(E)} \rho_E^{p-\d(E)-\ep}(\l)\,\rho_F^{q-\d(F)-\ep}(\l)\,(1-|\l|)\,\mu_E(d\l)=I_1.
\end{split}
\end{equation*}
By \eqref{set1.1}, one has $\rho_F(\l)>s$ as long as $\l\in D_t^c(E)$, so
\begin{equation}\label{inver2}
I_1\ge s^{q-\d(F)-\ep}\,\int\limits_{D_t^c(E)} \rho_E^{p-\d(E)-\ep}(\l)\,(1-|\l|)\,\mu_E(d\l).
\end{equation}

We apply \cite[Theorem 2]{fg1}, which claims that now
\begin{equation*}
\begin{split}
&{} \int_\bd \rho_E^{p-\d(E)-\ep}(\l)\,(1-|\l|)\,\mu_E(d\l) =\int\limits_{D_t(E)}\rho_E^{p-\d(E)-\ep}(\l)\,(1-|\l|)\,\mu_E(d\l) \\
&+ \int\limits_{D_t^c(E)}\rho_E^{p-\d(E)-\ep}(\l)\,(1-|\l|)\,\mu_E(d\l)=I_2+I_2^c=+\infty.
\end{split}
\end{equation*}
But $I_2<\infty$ thanks to the property of the Riesz measure, so $I_2^c=+\infty$. The relation \eqref{inver1} follows now from \eqref{inver2}.
\end{proof}


\begin{thebibliography}{99}


\bibitem{ahcla}
P. Ahern, D. Clark, On inner functions with $B^p$ derivatives, Michigan Math. J. {\bf 23} (1976), 107--118.

\bibitem{bla}
W. Blaschke, Eine Erweiterung des Satzes von Vitali \"uber Folgen analytischer Funktionen, S.-B. S\"acks Akad. Wiss. Leipzig
Math.-Natur. KI. {\bf 67} (1915), 194--200.

\bibitem{bgk1}
A. Borichev, L. Golinskii, S. Kupin, A Blaschke-type condition and its application to complex Jacobi matrices,
Bull. Lond. Math. Soc. {\bf 41} (2009), 117--123.

\bibitem{bgk2}
A. Borichev, L. Golinskii, S. Kupin, On zeros of analytic functions satisfying non-radial growth conditions,
Rev. Mat. Iberoam., {\bf 34}, no. 3 (2018), 1153--1176.

\bibitem{car}
L. Carleson, Sets of uniqueness for functions analytic in the unit disc, Acta Math., {\bf 87} (1952), 325--345.

\bibitem{fg1}
S. Favorov, L. Golinskii, A Blaschke-Type condition for Analytic and Subharmonic Functions
and Application to Contraction Operators, Amer. Math Soc. Transl. {\bf 226} (2009), 37--47.

\bibitem{fg2}
S. Favorov, L. Golinskii, Blaschke-type conditions for analytic and subharmonic functions in the
unit disk: local analogs and inverse problems, Comput. Methods Funct. Theory, {\bf 12} (2012) no. 1, 151--166.

\bibitem{Gar}
J. Garnett,  \textsl{Bounded analytic functions}, Graduate Texts in Mathematics, \textbf{236}, Springer, New York, 2007.

\bibitem{Ga-Ma}
J. Garnett, D. Marshall, \textsl{Harmonic measure}, Cambridge University Press, Cambridge, 2005.

\bibitem{HA}
B. Khabibullin, Z. Abdullina, A. Rozit, A uniqueness theorem and subharmonic test functions,
Algebra I Analiz, {\bf 30} (2018) no. 2, 318--334.

\bibitem{HT}
B. Khabibullin, N. Tamindarova, Subharmonic test functions and the distribution of zero sets of holomorphic functions,
Lobachevskii Journal of Math., {\bf 38} (2017) no. 1, 70--79.
 
\bibitem{KiSi03}
R. Killip, B. Simon, Sum rules for Jacobi matrices and their applications to spectral theory,
Ann. Math., {\bf 158} (2003), 253--321.


\bibitem{Li-Lo}
E. Lieb, M. Loss, \textsl{Analysis}, Graduate Studies in Mathematics, vol. 14, AMS, Providence, RI, 1997.

\bibitem{Ran}
T. Ransford, \textsl{Potential Theory in the Complex Plane}, London Math. Soc. Student Texts, vol. 28, Cambridge University Press, 1995.


\end{thebibliography}
\end{document}